\documentclass{amsart}

\usepackage[utf8]{inputenc}
\usepackage[english]{babel}
\usepackage{lmodern}
\usepackage[a4paper, margin=3cm]{geometry}
\usepackage{mathtools}
\usepackage{amsfonts}
\usepackage{amsthm}
\usepackage{amssymb}
\usepackage{listings}
\usepackage{graphicx}
\usepackage{tikz}
\usetikzlibrary{cd}
\usetikzlibrary{babel}
\usepackage{bbold}
\usepackage[colorinlistoftodos,textwidth=2cm]{todonotes}
\usepackage[font=small,labelfont=bf]{caption}
\usepackage[subrefformat=parens]{subcaption}
\usepackage{overpic}
\usepackage{enumerate}
\usepackage[T1]{fontenc}

\usepackage{latexsym}
\usepackage{pgf}
\usepackage{calrsfs,bbm}

\usepackage{amsmath}
\usepackage{theoremref}
\usepackage{mathrsfs}
\usepackage{amsthm}
\usepackage{mathtools}
\usepackage{enumitem}
\usepackage{tikz}
\usetikzlibrary{arrows,decorations.markings,knots,
	hobby,
	decorations.pathreplacing,
	shapes.geometric,
	calc}
\usepackage{float}
\usepackage{pgfkeys}

\usepackage{blkarray}
\usepackage{overpic}

\usepackage{enumitem}
\setlist{leftmargin=6.5mm}
\setenumerate[1]{label={(\roman*)}}
\setlist[enumerate]{leftmargin=8mm}

\graphicspath{ {images/} }
\usepackage{xcolor}
\usepackage[colorlinks,
    linkcolor={red!50!black},
    citecolor={blue!50!black},
    urlcolor={blue!80!black}]{hyperref}

\newcommand{\Z}{\mathbb{Z}}

\newcommand{\R}{\mathbb{R}}

\DeclareMathOperator{\lk}{lk}
\newcommand{\sgn}{\operatorname{sgn}}
\newcommand{\sign}{\operatorname{sign}}

\newcommand{\mult}{\operatorname{mult}}

\newtheorem{thm}{Theorem}
\newtheorem{lem}{Lemma}[section]
\newtheorem{prop}{Proposition}[section]
\newtheorem{cor}{Corollary}[section]

\theoremstyle{definition}

\newtheorem{conj}{Conjecture}
\newtheorem{exmp}{Example}[section]

\theoremstyle{remark}
\newtheorem{rem}{Remark}[section]

\makeatletter
\@namedef{subjclassname@1991}{2020 Mathematics Subject Classification}
\makeatother

 \title{On the Kashaev signature conjecture}

\author{David Cimasoni}
\address{David Cimasoni -- Section de math\'ematiques, Universit\'e de Gen\`eve, Suisse}
\email{david.cimasoni@unige.ch}

\author{Livio Ferretti}
\address{Livio Ferretti -- Section de math\'ematiques, Universit\'e de Gen\`eve, Suisse}
\email{livio.ferretti@unige.ch}

\begin{document}

\makeatletter
   \providecommand\@dotsep{5}
 \makeatother


\begin{abstract}
 In 2018, Kashaev introduced a square matrix indexed by the regions of a link diagram and conjectured that it provides a way of computing the Levine-Tristram signature and Alexander polynomial of the corresponding oriented link.
In this article, we show that for the classical signature (i.e. the Levine-Tristram signature at~$-1$), this conjecture follows from the work of Gordon-Litherland. We also relate Kashaev's matrix to Kauffman's ``Formal Knot Theory'' model of the Alexander polynomial. As a consequence, we establish the Alexander polynomial and classical signature parts of the conjecture for arbitrary links, as well as the full conjecture for definite knots.
\end{abstract}



\maketitle
	
	
	\section{Introduction}
	\label{sec:intro}
	
	The Levine-Tristram signature~$\sigma_L$ and Alexander polynomial~$\Delta_L$ of an oriented link~$L$ in
	the 3-sphere~$S^3$ are among the most studied and best understood link invariants.
	They can be defined as follows.
	Let~$F$ be a Seifert surface for~$L$, i.e. an oriented connected compact surface~$F$ smoothly embedded
	in~$S^3$ with oriented boundary~$\partial F=L$. Let
	\[
	\alpha\colon H_1(F;\Z)\times H_1(F;\Z)\longrightarrow \Z
	\]
	 be the associated {\em Seifert form\/}, i.e. the bilinear map defined by~$\alpha([x],[y])=\lk(x^-,y)$, where~$\lk$ stands for the linking number, and~$x^-\subset S^3\setminus F$ denotes the cycle~$x\subset F$
	pushed in the negative normal direction off~$F$. Writing~$A$ for an associated matrix and
	fixing~$\omega\in S^1\setminus\{1\}$,  the matrix
	\[
	H(\omega):=(1-\omega)A+(1-\overline{\omega})A^T
	\]
	is Hermitian and therefore has a well-defined signature, namely the number of positive eigenvalues minus the
	number of negative ones. The {\em Levine-Tristram signature\/} of~$L$~\cite{levine_knot_1969,Tri69} is the map
	\[
	\sigma_L\colon S^1\setminus\{1\}\longrightarrow \Z
	\]
	given by~$\sigma_L(\omega)=\sign H(\omega)$.
	This map is a well-defined link invariant, i.e. depends neither on the choice of the Seifert surface~$F$,
	nor on the choice of a basis of its homology (see e.g.~\cite{lickorish_introduction_1997}).
	The same methods can be used to show that the {\em Alexander polynomial\/}
	\[
	\Delta_L(t)=\det\big(t^{1/2}A-t^{-1/2}A^T\big)\in\Z[t^{\pm 1/2}]
	\]
	is also a well-defined invariant of the oriented link~$L$. (This normalized version is often referred to as the {\em Alexander-Conway polynomial\/} of~$L$~\cite{alexander_topological_1928,Con70,Kau81}).
	
	These invariants enjoy numerous alternative definitions, most of them dating back several decades
	(see~\cite{conway_levine-tristram_2019} for a survey of the Levine-Tristram signature).
	It therefore came as a surprise when, in a recent attempt to understand the metaplectic invariants of
	Goldschmidt-Jones~\cite{Goldman-Jones},
	Kashaev conjectured a novel way to compute the Levine-Tristram signature and
	Alexander polyomial of an oriented link~\cite{kashaev_symmetric_2021}. We now recall his construction.
	
	\medskip
	
	To any oriented link diagram~$D$, Kashaev associated a matrix~$\tau_D$ indexed by the regions of~$D$
	with coefficients in the polynomial ring~$\Z[x]$. It is defined as the sum over all crossings~$c$ of~$D$ of the matrix
\begin{equation}
\label{eq:tau}
	\tau_c=\sgn(c)\;\;\begin{blockarray}{ccccc}
i & j & k & \ell \\
\begin{block}{(cccc)c}
  2x^2-1 & x & 1 & x & i\\ x & 1 & x & 1 &j\\ 1 & x & 2x^2-1 & x & k \\ x & 1 & x & 1 & \ell\\
\end{block}
\end{blockarray}\,,
\end{equation}
where~$\sgn(c)=\pm 1$ denotes the sign of~$c$, and the regions~$i,j,k,\ell$ around the crossing~$c$ are labeled as illustrated in Figure~\ref{fig:regions}. Equation~\eqref{eq:tau} should be understood as describing the non-vanishing coefficients of a matrix indexed by the regions of~$D$; also, if the diagram is not {\em reduced\/}, i.e. if the regions~$i,j,k,\ell$ around a crossing~$c$ are not all distinct, then one should add the corresponding rows and columns of~$\tau_c$.
One easily checks that the coefficients of~$\tau_D$ actually belong to the
ring~$\Z[2x]$.
	\begin{figure}[H]
	\centering
\begin{overpic}[width=2cm]{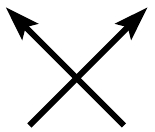}
 \put (50,65){$i$}
    \put (15,35){$j$}
        \put (50,0){${k}$}
        \put (85,35){${\ell}$}
    \end{overpic}
    \caption{Labeling of the regions in Equation~\eqref{eq:tau}}
    \label{fig:regions}
		\end{figure}
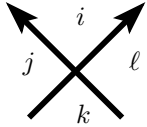

 \begin{exmp}
 	Consider the positive trefoil knot, with diagram~$D$ and regions numbered as follows.
 	\begin{center}
 		\begin{tikzpicture}[use Hobby shortcut,scale=0.7,add arrow/.style={postaction={decorate}, decoration={
 				markings,
 				mark=at position 0.5 with {\arrow[line width=1pt]{>}}}}]
 		\begin{knot}[
 		consider self intersections=true,
 		ignore endpoint intersections = false,
 		flip crossing/.list={3,5}
 		]
 		\strand[very thick] ([closed]90:2) foreach \k in {1,...,3} { .. (90-360/3+\k*720/3:1) .. (90+\k*720/3:2) } (90:2);
 		\end{knot}
 		\node at (0,0) []{1};
 		\node at (-1.5,1) []{2};
 		\node at (0,1.4) []{3};
 		\node at (1.3,-0.5) []{4};
 		\node at (-1.3,-0.5) []{5};
 		\end{tikzpicture}
 		\end{center}
 The associated Kashaev matrix is given by
 \[
 \tau_D = \begin{pmatrix} 
 	3 & 3 & 2x & 2x & 2x \\ 3 & 3 & 2x  & 2x & 2x \\ 2x & 2x & 4x^2-2 & 1 & 1 \\ 2x & 2x & 1 & 4x^2-2 & 1 \\ 2x & 2x & 1 & 1 & 4x^2-2 \\
 	\end{pmatrix}\,.
 \]
 \end{exmp}

\medskip

Kashaev then studied the effect of Reidemeister moves on~$\tau_D$, obtaining the following result: given any oriented link diagram~$D$ and any~$x\in\R$, the integer
\[
\sign(\tau_D[x]) - w(D)
\]
is a link invariant, where~$\sign(\tau_D[x])$ is the signature of the symmetric matrix~$\tau_D$ evaluated at $x\in\mathbb{R}$, and~$w(D)$ is the {\em writhe\/} of~$D$, namely the number of positive crossings minus the number of
negative ones. Even though not stated formally in~\cite{kashaev_symmetric_2021}, it also follows from this
investigation that the torsion of the~$\Z[2x]$-module presented by~$\tau_D$ is an invariant.
(In fact, the isomorphism class of the full module is an invariant provided~$D$ is connected).
In particular, its determinantal ideals yield invariant polynomials (up to multiplication by a unit of~$\Z[2x]$, i.e. up to sign).
Kashaev conjectured that these are well-known invariants.
	
\begin{conj}[Kashaev~\cite{kashaev_symmetric_2021}]
\label{conj}
	Let~$D$ be an oriented diagram for an oriented link~$L$.
	\begin{enumerate}
	\item If~$\widetilde{\tau}_D$ denotes the matrix~$\tau_D$ evaluated at~$2x=t^{1/2}+t^{-1/2}$ with two
	rows and columns corresponding to two adjacent regions removed, then we have
			\[
			\det(\widetilde{\tau}_D) = \pm\Delta_L(t)^2\in\Z[t^{\pm 1/2}]\,.
			\]
	\item For any~$\omega \in S^1\setminus\{1\}$, we have
		\[
		\sign(\tau_D[x]) - w(D)=2 \sigma_L(\omega)\,,
		\]
		where~$2x = \omega^{1/2}+\omega^{-1/2} \in \mathbb{R}$.
	\end{enumerate}
\end{conj}

Note that only the second point of this conjecture was explicitely stated in~\cite{kashaev_symmetric_2021}.
Nevertheless, the first point is a rather obvious guess from the examples computed in~\cite{kashaev_symmetric_2021}, and was discussed by the authors with Rinat Kashaev, hence
the attribution.
	
	\medskip
	
	 In the present article, we prove the first point of this conjecture in full generality, the second point for the {\em classical signature\/}~$\sigma(L)=\sigma_L(-1)$~\cite{Tro62}, and the full conjecture for {\em definite knots\/},
	 namely knots admitting a Seifert matrix~$A$ such that~$A+A^T$ is (positive or negative) definite. In other words, we have the following result.
	
	\begin{thm}\label{thm}
		 Let~$D$ be an oriented diagram for an oriented link~$L$.
		 \begin{enumerate}
			\item 
			The matrix~$\widetilde{\tau}_D$ satisfies~$\det(\widetilde{\tau}_D) = \pm\Delta_L(t)^2$.
			\item We have the equality~$\sgn(\tau_D[0]) - w(D)=2 \sigma(L)$.
			\item If~$L$ is a definite knot, then the equality~$\sgn(\tau_D[x]) - w(D)=2 \sigma_L(\omega)$
			holds for all~$\omega\in S^1\setminus\{1\}$, with~$2x= \omega^{1/2}+\omega^{-1/2} $.
		\end{enumerate}
	\end{thm}

Our proof of the first point relies on a relation between the Kashaev matrix~$\tau_D$ and
Kauffman's ``Formal Knot Theory'' model of the Alexander polynomial~\cite{kauffman_formal_1983,kauffman2006remarks}.
As for the second point, it follows from relating~$\tau_D[0]$ with two copies of the Goeritz matrix
and harnessing the  Gordon-Litherland formula for the classical signature~\cite{gordon_signature_1978}. 
These two results then imply the third one in a rather straightforward way (and actually yield the conjecture for a wider class of knots; see Remark~\ref{rem:ext}).

Therefore, the present article not only provides a proof of parts of the Kashaev conjecture. Its purpose is also
to show that this conjecture can be understood as a reformulation of Kauffman's model for~$\Delta_L$ together with a rather surprising extension of the Gordon-Litherland theorem from the classical signature~$\sigma(L)$ to the full Levine-Tristram signature~$\sigma_L$ (see Remarks~\ref{rem:Kauffman} and~\ref{rem:G-L}).

\medskip

This paper is organised as follows. In Section~\ref{sec:prelim}, we recall the necessary background
on the aforementioned Kauffman model for~$\Delta_L$ (Section~\ref{paragraph_Kauffman_alex})
and Gordon-Litherland formula for~$\sigma(L)$ (Section~\ref{paragraph_gordon_litherland}).
Section~\ref{sec:proof} contains the proof of Theorem~\ref{thm}, each of the three points being dealt
with in an individual subsection.

\subsection*{Acknowledgments}
The authors would like to thank Sebastian Baader, Pierre Bagnoud, Anthony Conway, Livio Liechti and Rinat Kashaev for useful discussions, as well as the anonymous referee for sensible comments on an earlier version of this article.
Support from the Swiss NSF grant 200021-212085 is thankfully acknowledged.

	\section{Preliminaries}
	\label{sec:prelim}
	
This section deals with the preliminaries to the proof of Theorem~\ref{thm}. We start in Section~\ref{paragraph_Kauffman_alex} by recalling Kauffman's model for~$\Delta_L$, while
Section~\ref{paragraph_gordon_litherland} contains a brief presentation of the Gordon-Litherland
formula for~$\sigma(L)$.
	
	\subsection{The Kauffman model for the Alexander polynomial}\label{paragraph_Kauffman_alex}
	
	One of the main points of Kauffman's {\em Formal Knot Theory\/} treatise~\cite{kauffman_formal_1983} (see also~\cite{kauffman2006remarks}) is the construction of a combinatorial model for the normalized Alexander-Conway polynomial.  It can be summarized as follows.
	
Given an oriented diagram~$D$ of a link~$L$, Kauffman defined a matrix~$K(D)$ whose rows are indexed by the crossings of~$D$ and whose columns are indexed by the regions of~$D$. For any region~$i$ and crossing~$c$ of~$D$, the coefficient~$K(D)_{ci}$ is the label of the corner corresponding to the region~$i$ at the crossing~$c$, as in Figure~\ref{figure_labels_kauffman}.
(If a region abuts a corner from two sides, then the corresponding labels should be added.)
Writing~$\widetilde{K}(D)$ for the matrix obtained from~$K(D)$ by deleting two columns corresponding to adjacent regions, Kauffman proved that~$\det\widetilde{K}(D)$ is an invariant of~$L$ up
to sign, and satisfies
\begin{equation}\label{eq:Kauffman}
\det\widetilde{K}(D) =\pm\Delta_K(t)\in\Z[t^{\pm 1/2}]\,.
\end{equation}

	\begin{figure}[H]
		\centering
		\begin{tikzpicture}
			[scale=1.5, line width=3pt]  
			\node at (-1.5,0.5) []{$t^{1/2}$};
			\node at (-2.5,0.5) []{$1$};
			\node at (-2.5,-0.5)[]{$t^{-1/2}$};
			\node at (-1.5,-0.5)[]{$1$};
			\node at (2.5,0.5) []{$t^{-1/2}$};
			\node at (1.5,0.5) []{$1$};
			\node at (1.5,-0.5)[]{$t^{1/2}$};
			\node at (2.5,-0.5)[]{$1$};
			\draw (-3,0) edge [->, -stealth] (-1,0) {};
			\draw [dashed,dash pattern=on 35pt off 15pt] (-2,-1) edge [->,-stealth] (-2,1);
			\draw [dashed,dash pattern=on 35pt off 15pt] (1,0) edge [->, -stealth] (3,0) {};
			\draw  (2,-1) edge [->,-stealth] (2,1);	
		\end{tikzpicture}
		\caption{Kauffman's labels}
		\label{figure_labels_kauffman}
	\end{figure}
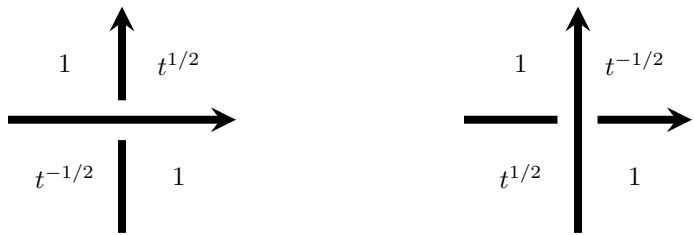
	
	\begin{rem}\label{remark_congruence}
As one easily checks, the matrix~$K(D)$ can be transformed to the matrix~$\big(\widetilde{K}(D) \;\; 0 \;\; 0 \big)$
by adding to the two last columns (corresponding to adjacent regions) linear combinations of the others.
	\end{rem}
	
	\subsection{The Gordon-Litherland formula for the classical signature}\label{paragraph_gordon_litherland}
	
	In~\cite{gordon_signature_1978}, Gordon and Litherland defined a quadratic form associated to any (non-necessarily orientable) spanning surface of a link~$L$. Using 4-dimensional techniques, they showed how this form relates the classical signature~$\sigma(L)$ to a Goeritz matrix of~$L$~\cite{Goeritz}, thus providing a simple, diagrammatic way of computing the signature. We now briefly outline this result.
	
Consider an oriented link~$L$ and an oriented diagram~$D$ for~$L$. Color the regions of $\R^2\setminus D$ with two colours, say black and white (denoted by~$\sf{b}$ and~$\sf{w}$),  in a checkerboard manner. Then, for any choice~$\sf{v}\in\{b,w\}$ among these two colours, one can associate to a crossing~$c$ of~$D$ two signs~$\eta_{\sf{v}}(c)$ and~$t_{\sf{v}}(c)$ as described in Figure~\ref{figure_eta_t}. Note that~$\eta_{\sf{v}}(c)$ depends on the under/over crossing information but not on the orientation, while~$t_{\sf{v}}(c)$ depends on the orientation, but not on the under/over crossing.
	
	\begin{figure}[htb]
		\begin{minipage}[b]{0.4\textwidth}
				\centering
				\begin{tikzpicture}
					[scale=1.7, line width=3pt]
					\fill[color=blue!30]             (-1.5,0) -- (-1,1) -- (-0.5,0) -- cycle;
					\fill[color=blue!30]             (-1.5,2) -- (-1,1) -- (-0.5,2) -- cycle;
					\draw (-0.5,0) -- (-1.5,2) node [] {};
					\draw [dashed,dash pattern=on 49pt off 10pt] (-1.5,0) edge (-0.5,2);
					\node at (-1,-0.3) []{$\eta_{\sf{v}}(c) =-1$};
					\fill[color=blue!30]             (0.5,0) -- (1,1) -- (1.5,0) -- cycle;
					\fill[color=blue!30]             (0.5,2) -- (1,1) -- (1.5,2) -- cycle;
					\draw [dashed,dash pattern=on 49pt off 10pt] (1.5,0) edge (0.5,2);
					\draw  (0.5,0) edge (1.5,2);
					\node at (1,-0.3) []{$\eta_{\sf{v}}(c) =1$};	
				\end{tikzpicture}
		\end{minipage}
		\hspace{0.15\textwidth}
		\begin{minipage}[b]{0.4\textwidth}
				\centering
				\begin{tikzpicture}
					[scale=1.7, line width=3pt,  startarrow/.style 2 args={
						decoration={             
							markings, 
							mark=at position 0.3 with {\arrow{triangle 45}, \node[#1] {#2};}
						},
						postaction={decorate}
					},endarrow/.style 2 args={
						decoration={             
							markings, 
							mark=at position 0.91 with {\arrow{triangle 45}, \node[#1] {#2};}
						},
						postaction={decorate}
					}]
					\fill[color=blue!30]             (-1.5,0) -- (-1,1) -- (-0.5,0) -- cycle;
					\fill[color=blue!30]             (-1.5,2) -- (-1,1) -- (-0.5,2) -- cycle;	
					\draw [startarrow](-0.5,0) -- (-1.5,2);
					\draw [startarrow](-1.5,0) -- (-0.5,2);
					\node at (-1,-0.3) []{$t_{\sf{v}}(c) =1$};
					\fill[color=blue!30]             (0.5,0) -- (1,1) -- (1.5,0) -- cycle;
					\fill[color=blue!30]             (0.5,2) -- (1,1) -- (1.5,2) -- cycle;		
					\draw [startarrow](1.5,0) -- (0.5,2);
					\draw [endarrow](1.5,2) -- (0.5,0);
					\node at (1,-0.3) []{$t_{\sf{v}}(c) =-1$};	
				\end{tikzpicture}
		\end{minipage}
		\caption{Definition of $\eta_{\sf{v}}(c)$ and $t_{\sf{v}}(c)$, with ${\sf{v}}$ the shaded region}
		\label{figure_eta_t}
	\end{figure}
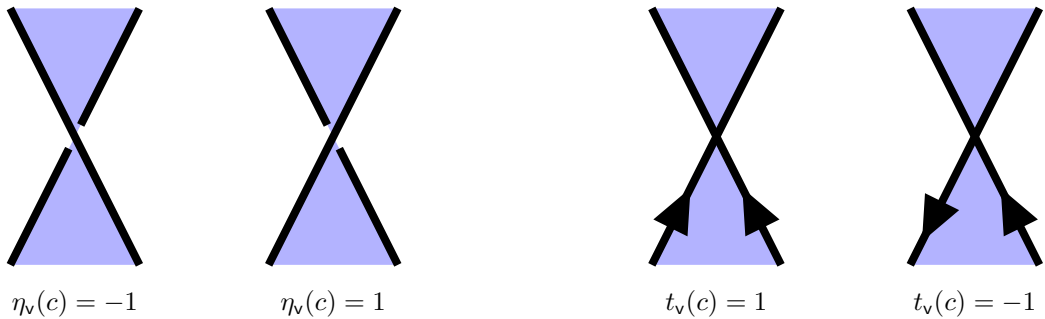
	
The \textit{Goeritz matrix} associated to the colour~${\sf{v}}$ is the matrix~$G_{\sf{v}}(D) = (g_{ij})_{i,j}$ indexed by the regions of colour~${\sf{v}}$ and defined by
\[
g_{ij} = \sum_{c\sim i,\,c\sim j} \eta_{\sf{v}}(c)
\]
for~$i\neq j$, where the sum is over all crossings~$c$ incident to both regions~$i$ and~$j$, and by
\[
g_{ii}=- \sum_{k\neq i} g_{ik}\,.
\]
A Goeritz matrix is always symmetric, so its signature is well defined, but it is not invariant under Reidemeister moves. In order to produce an invariant, one needs to consider the correction term
	\[
	\mu_{\sf{v}}(D) = \sum\limits_{c\,:\, t_{\sf{v}}(c)=-1} -\eta_{\sf{v}}(c)\,,
	\]
	the sum being over all the crossings with~$t_{\sf{v}}(c)=-1$. Gordon and Litherland proved that for any checkerboard colour~${\sf{v}}$, we have
	\begin{equation}\label{eq:GL}
	\sign G_{\sf{v}}(D) - \mu_{\sf{v}}(D) = \sigma(L)\,.
	\end{equation}

	\begin{rem}
	The sum of all the columns in a Goeritz matrix~$G$ is equal to zero, so~$G$ is congruent to $(0)\oplus\widetilde{G}$, where $\widetilde{G}$ is obtained from~$G$ by deleting the first row and first column.
The matrix~$\widetilde{G}$ is often referred to as the Goeritz matrix, while~$G$ is sometimes called {\em pre-Goeritz\/}. However, this change is irrelevant in the computation of the signature, and it will be more practical
for us to consider the entire matrix.
	\end{rem}
	
	\section{Proof of Theorem~\ref{thm}}
	\label{sec:proof}
	
	This section contains the proof of Theorem~\ref{thm}: Section~\ref{sub:Alex} deals with the first point,
	Section~\ref{sub:sign} with the second, and Section~\ref{sub:definite} with the third.

	\subsection{The Alexander polynomial and signature jumps}
	\label{sub:Alex}
	The aim of this section is to prove the first point of Theorem~\ref{thm}, and to derive a lemma
	about the behavior of the Kashaev signature.
	
	Let us fix an oriented diagram~$D$ for an oriented link~$L$, and first focus on the relation between~$\tau_D=\tau_D[x]$ and~$\Delta_L$.
	To do so, we fix an ordering of the regions of~$D$ and of its crossings, denoted by~$c_1,\cdots,c_n$.
	Recall the associated Kauffman matrix~$K(D)$ defined in Section~\ref{paragraph_Kauffman_alex}.

	\begin{prop}\label{proposition_alexander}
	If~$D$ is a diagram for an oriented link~$L$, then we have the equality
	\[
	\tau_D[x] = K(D)^T S(D) K(D),\,
	\]
	where~$2x = t^{1/2} + t^{-1/2}$ and~$S(D)$ stands for the diagonal matrix~$\operatorname{diag}(\sgn(c_1),\cdots,\sgn(c_n))$.
	\end{prop}
	
	\begin{proof}
The proof consists in expanding the matrix on the right-hand side of the equality (that we write~$K^T S K$ for simplicity), and comparing it to~$\tau_D$
evaluated at~$2x = t^{1/2} + t^{-1/2}$. Given any regions~$i,j$ of~$D$, we have
\[
\left(K^T S K\right)_{ij}=\sum_{c\sim i,\,c\sim j}\sgn(c) K_{ci}K_{cj}\,,
\]
where the sum is over all the crossings of~$D$ incident to the region~$i$ and to the region~$j$, and the coefficients of~$K$ are given by the labels of Figure~\ref{figure_labels_kauffman}.
	
If~$i$ and~$j$ are different regions of the same colour, then each crossing~$c$ incident to both~$i$ and~$j$
contributes~$\sgn(c)$ to the coefficient~$\left(K^T S K\right)_{ij}$. This is precisely the contribution of~$c$ to~$\left(\tau_D\right)_{ij}$ (recall Equation~\eqref{eq:tau}), so this case is checked.
		
Let us now assume that~$i$ and~$j$ are two regions of different colours. Then, every edge of~$D$ adjacent to~$i$ and~$j$ gives two contributions to~$\left(K^T S K\right)_{ij}$, one for each crossing adjacent to the edge: these contributions sum up to~$t^{1/2}+t^{-1/2}=2x$ (resp. $-t^{1/2}-t^{-1/2}=-2x$) if both crossings are positive (resp. negative), and vanish if the crossings have different signs. This coincides with the contributions of these two crossings to~$\left(\tau_D\right)_{ij}$.

Finally, let us assume that~$i=j$. If the orientation of~$D$ near a crossing~$c$ incident to~$i$ yields a coherent orientation of the region~$i$, then the crossing~$c$ contributes~$\sgn(c)$ to the~$i^\text{th}$ diagonal coefficient of both matrices~$\tau_D$ and~$K^T S K$.
If this is not the case, then the crossing~$c$ yields different polynomial coefficients to both these matrices.
Note that the number of contributions of crossings incident to the fixed region~$i$ giving such different coefficients is always even. (If the diagram is reduced, the number of such crossings is even, but in general,
one must consider crossing contributions.) Hence, we can start from one such contribution, move along the boundary of the region~$i$ in a fixed direction, and group them by pairs of consecutive such contributions.
The sum of two consecutive contributions is~$t + t^{-1}=2(2x^2-1)$ (resp. $-t - t^{-1}=-2(2x^2-1)$) if both crossings are positive (resp. negative) and $0$ if the crossings have different signs.
Once again, this coincides with the sum of these two contributions to~$\left(\tau_D\right)_{ii}$.
\end{proof}

\begin{proof}[Proof of Theorem~\ref{thm}~(i)]
Let~$D$ be a diagram for an oriented link~$L$, and let~$\widetilde{K}(D)$ denote the Kauffman matrix~$K(D)$ with two columns corresponding to adjacent regions removed. By Proposition~\ref{proposition_alexander}, we have
\begin{equation}
\label{eq:K-K}
\widetilde{\tau}_D=\widetilde{K}(D)^TS(D)\widetilde{K}(D)\,.
\end{equation}
The equality~$\det\widetilde{\tau}_D=\pm\Delta_L(t)^2$ now follows from Kauffman's model in the form of Equation~\eqref{eq:Kauffman}.
\end{proof}
	
\begin{rem}\label{rem:Kauffman}
Note that the relation~\eqref{eq:K-K} between the Kauffman and Kashaev matrices not only allows us
to prove the Alexander polynomial part of Kashaev's conjecture. It also shows that this latter statement is
actually equivalent to Kauffman's result. Therefore, an independent proof of the first point of Conjecture~\ref{conj} would automatically provide a new proof of Kauffman's model for the Alexander polynomial.
\end{rem}

Applying Remark~\ref{remark_congruence}, we immediately get the following corollary.
	
\begin{cor}\label{corollary_congruence_kauffman}
The matrix~$\tau_D$ is congruent to~$\widetilde{\tau}_D\oplus(0)^{\oplus 2}$, with~$\widetilde{\tau}_D=\widetilde{K}(D)^T S(D)\widetilde{K}(D)$.\qed
\end{cor}

Using this relation to the Alexander polynomial, we now study the jumps of Kashaev's signature, defined as
\[
J^{\pm}(x) := \pm (\lim\limits_{y\to x^{\pm}} \sign(\tau_D[y]) - \sign(\tau_D[x]))\,.
\]
To do so, we denote by~$\mult_\omega(\Delta_L)$ the multiplicity of~$\omega\in S^1$ as a root of the polynomial~$\Delta_L$.
	
	\begin{lem}\label{proposition_kashaev_jumps}
		Let~$L$ be a link with non-vanishing Alexander polynomial~$\Delta_L$, and let~$D$ be a diagram of~$L$.
		After the change of variable $2x = \omega^{1/2}+\omega^{-1/2}$ with~$\omega\in S^1$, the signature~$\sign(\tau_D[x])$ becomes a step function on~$S^1$ which can have discontinuities only at roots of~$\Delta_L(t)$, and whose jumps satisfy~$|J^{\pm}(\omega)| \leq 2\mult_\omega(\Delta_L)$.
	\end{lem}
	
	\begin{proof}
		By Corollary \ref{corollary_congruence_kauffman}, it is enough to study the jumps of the signature of~$\widetilde{\tau}_D=\widetilde{K}(D)^T S(D)\widetilde{K}(D)$. Moreover, since~$\det\widetilde{\tau}_D=\pm \Delta_L^2 \neq 0$, it follows that the signature of $\widetilde{\tau}(D)$ can jump only at the roots of~$\Delta_L$, and that the jumps~$|J^{\pm}(\omega)|$ are bounded by the nullity of~$\widetilde{\tau}_D$ at~$t=\omega$.
Let us consider~$\widetilde{\tau}_D$ as a matrix with coefficients in~$\mathbb{R}[t^{\pm1/2}]$. Since this ring is a principal ideal domain, there exist matrices~$P,Q \in\operatorname{GL}(\mathbb{R}[t^{\pm1/2}])$ such that
		\[
		P\,\widetilde{\tau}_D\,Q = \begin{pmatrix}
			d_1 & & \\ & \ddots & \\ & & d_n\\
		\end{pmatrix}
		\]
with $d_i \in \mathbb{R}[t^{\pm1/2}]$. In particular, we have~$d_1\cdots d_n =\Delta_L^2$ up to multiplication
by units of~$\mathbb{R}[t^{\pm1/2}]$.
But the nullity of~$\widetilde{\tau}_D$ at~$t=\omega$ is the number of~$d_i$ such that~$d_i(\omega) = 0$, which is bounded by~$\mult_\omega(\Delta_L^2).$  Therefore, we get the inequality~$|J^{\pm}(\omega)| \leq\mult_\omega(\Delta_L^2) = 2\mult_\omega(\Delta_L)$.		
	\end{proof}
	
	\subsection{The classical signature}
	\label{sub:sign}
	
	We now study the relation of Kashaev's matrix with the classical signature $\sigma(L)=\sigma_L(-1)$. Under the usual change of variables~$2x = t^{1/2} + t^{-1/2}$, this corresponds to studying~$\tau_D$ at~$x=0$.
	
Consider an oriented link diagram~$D$ whose regions are coloured in a checkerboard manner with two colours~$\sf{b}$ and~$\sf{w}$. For any colour $\sf{v} \in \{\sf{b},\sf{w}\}$ and any crossing~$c$, recall the signs~$\eta_{\sf{v}}(c)$ and $t_{\sf{v}}(c)$ defined in Figure~\ref{figure_eta_t}. The proof of the following result is immediate.
	
	\begin{lem}\label{lemma_sign}
		For any colour $\sf{v}$ and crossing $c$, we have the equality~$\eta_{\sf{v}}(c) t_{\sf{v}}(c) = \sgn(c)$.\qed
	\end{lem}
	
Recall the definition of the correction term~$\mu_{\sf{v}}(D) = -\sum_{c\,:\,t_{\sf{v}}(c)=-1}\eta_{\sf{v}}(c)$.
	
	\begin{lem} \label{lemma_writhe}
	For any diagram~$D$, we have~$\mu_{\sf{w}}(D) + \mu_{\sf{b}}(D) = w(D)$.
	\end{lem}
	
	\begin{proof}
		The definition of~$\mu_{\sf{v}}$ together with Lemma~\ref{lemma_sign} yield
		\begin{align*}
			\mu_{\sf{w}}(D) + \mu_{\sf{b}}(D) = & \sum\limits_{c\,:\,t_{\sf{w}}(c)=-1} -\eta_{\sf{w}}(c) + \sum\limits_{c\,:\,t_{\sf{b}}(c)=-1} -\eta_{\sf{b}}(c) \\
			& = \sum\limits_{c\,:\,t_{\sf{w}}(c)=-1} -\eta_{\sf{w}}(c) + \sum\limits_{c\,:\,t_{\sf{w}}(c)=1} \eta_{\sf{w}}(c) \\
			& = \sum\limits_{c\,:\,t_{\sf{w}}(c)=-1} t_{\sf{w}}(c)\eta_{\sf{w}}(c) + \sum\limits_{c\,:\,t_{\sf{w}}(c)=1} t_{\sf{w}}(c)\eta_{\sf{w}}(c) \\
			& = \sum\limits_{c\,:\,t_{\sf{w}}(c)=-1} \sgn(c) + \sum\limits_{c\,:\,t_{\sf{w}}(c)=1} \sgn(c) = w(D)\,.\qedhere
		\end{align*}
	\end{proof}
	
	Let us now consider the matrix~$\tau_D[x]$ evaluated at~$x=0$.  Note that	if two regions~$i$ and~$j$ have different colours,
	then the corresponding coefficient~$\tau_D[0]_{ij}$ vanishes.
	Hence, we see that~$\tau_D[0]$ splits as the direct sum
	\[
	\tau_D[0] = \tau_{\sf{w}}(D) \oplus \tau_{\sf{b}}(D)\,,
	\]
	where the matrix~$\tau_{\sf{w}}(D)$ (resp. $\tau_{\sf{b}}(D)$) is indexed by the white (resp. black) regions of~$D$. For any colour~$\sf{v}\in\{\sf{w},\sf{b}\}$, the definition of~$\tau_D$ and Lemma~\ref{lemma_sign} yield
\[
\tau_{\sf{v}}(D)_{ij} = \begin{cases}
		\sum\limits_{c\sim i,\, c\sim j} \eta_{\sf{v}}(c)t_{\sf{v}}(c), & \text{ if } i\neq j \\
		\sum\limits_{c \sim i}  -\eta_{\sf{v}}(c), & \text{ if } i=j,
	\end{cases}
	\]
where the first (resp.  second) sum is over all crossings incident to both regions~$i$ and~$j$ (resp. incident to the region~$i$).
	Moreover, given two regions~$i$ and~$j$ of the same colour~$\sf{v}$, for any crossing~$c$ incident to~$i$ and~$j$, the sign~$t_{\sf{v}}(c)$ depends only on~$i$ and~$j$, and can therefore be denoted by~$t_{ij}$. We get
\[
\tau_{\sf{v}}(D)_{ij} = \begin{cases}
		t_{ij}\sum\limits_{c\sim i,\, c\sim j} \eta_{\sf{v}}(c), & \text{ if } i\neq j \\
		\sum\limits_{c \sim i}  -\eta_{\sf{v}}(c), & \text{ if } i=j\,.
	\end{cases}
\]
Comparing this with the definition of the Goeritz matrices in Section~\ref{paragraph_gordon_litherland}, we see that~$\tau_D[0]$ is very close to the direct sum~$G_{\sf{w}}(D) \oplus G_{\sf{b}}(D)$: the only differences are the signs~$t_{ij}$ that appear in~$\tau_{\sf{v}}(D)$ but not in~$G_{\sf{v}}(D)$.

Everything is now set up for the proof of the conjecture for the classical signature.
	
	\begin{proof}[Proof of Theorem~\ref{thm}~(ii)]
Let~$D$ be a diagram for an oriented link~$L$. Applying twice the Gordon-Litherland formula~\eqref{eq:GL} together with Lemma~\ref{lemma_writhe}, we have
\[
2 \sigma(L) = \sign G_{\sf{w}}(D) - \mu_{\sf{w}}(D) + \sign G_{\sf{b}}(D) - \mu_{\sf{b}}(D) = \sign G_{\sf{w}}(D) + \sign G_{\sf{b}}(D) - w(D)\,.
\]
Therefore, we only need to find a congruence between~$\tau_D[0] = \tau_{\sf{w}}(D) \oplus \tau_{\sf{b}}(D)$ and $G_{\sf{w}}(D) \oplus G_{\sf{b}}(D)$.
		
To simplify the arguments, we now use the well-known fact that every link admits a {\em special diagram\/} (see e.g.~\cite[Proposition 13.14]{burde_knots_2003}). This can be understood as a diagram such that for every colour~$\sf{v}$, the sign~$t_{\sf{v}}(c)$ does not depend on~$c$. Taking~$D$ special and choosing the checkerboard colouring such that~$t_{\sf{w}}(c)=1$ for all~$c$,  we get~$\tau_{\sf{w}}(D) = G_{\sf{w}}(D)$ and
\[
\tau_{\sf{b}}(D)_{ij} = \begin{cases}
		\sum\limits_{c\sim i,\, c\sim j} -\eta_{\sf{b}}(c), & \text{ if } i\neq j \\
		\sum\limits_{c \sim i}  -\eta_{\sf{b}}(c), & \text{ if } i=j\,.
	\end{cases}
	\]
We are now left with proving that~$\tau_{\sf{b}}(D)$ is congruent to~$G_{\sf{b}}(D)$.
Since we have~$t_{\sf{w}}(c)=1$ for all~$c$, any white region abuts an even number of crossings.
It follows that the adjacency graph associated to the black regions is a bipartite graph.
This implies the equality~$\tau_{\sf{b}}(D) = P^T G_{\sf{b}}(D) P$, with~$P$ the diagonal matrix with diagonal
coefficient equal~$1$ or~$-1$ according to whether the corresponding black region belongs to one set of the bipartition or the other. 
	\end{proof}

	\begin{rem}\label{rem:G-L}
Note that the relationship between~$\tau_D[0]$ and the direct sum of two Goeritz matrices not only allows us
to prove the classical signature part of Kashaev's conjecture. It also shows that this latter statement is equivalent to the Gordon-Litherland formula, provided one knows that~$\sign G_{\sf{v}}(D) - \mu_{\sf{v}}(D)$ is an invariant (an easy fact already established by Goeritz~\cite{Goeritz}). Therefore, the second point of Conjecture~\ref{conj} should be understood
as an extension of the Gordon-Litherland formula from the classical signature to the full Levine-Tristram signature.
	\end{rem}

	\subsection{The signature conjecture for definite knots}
	\label{sub:definite}
	
In this final section, we conclude the proof of Theorem~\ref{thm} by establishing the signature conjecture for definite knots. To do so, let us start by focusing on the behaviour of~$\tau_D[x]$ at $x=1$.

\begin{lem}\label{proposition_signature_writhe}
If~$D$ is a diagram for an oriented knot~$K$, then~$\sgn(\tau_D[1])=w(D)$.
\end{lem}
	
	\begin{proof}
		By Corollary~\ref{corollary_congruence_kauffman}, the matrix~$\tau_D$ is congruent to~$\widetilde{\tau}_D\oplus(0)^{\oplus 2}$ with
\[
\widetilde{\tau}_D=\widetilde{K}(D)^T S(D)\widetilde{K}(D)\,,
\]
so~$\sign(\tau_D[x])=\sign(\widetilde{\tau}_D[\omega])$ for all~$x\in\R$ and~$\omega\in S^1$ with~$2x=\omega^{1/2}+\omega^{-1/2}$. Evaluating Equation~\eqref{eq:Kauffman} at~$x=1$, which corresponds to~$\omega=1$, we get that~$\widetilde{K}(D)[1]$ is an integer-valued matrix whose determinant
is equal to~$\pm\Delta_K(1)$. Since~$K$ is a knot, this value is non-zero. Therefore, by Silvester's law of inertia, we now have
\[
\sign(\tau_D[1])=\sign(\widetilde{\tau}_D[1])=\sign(\widetilde{K}(D)^T S(D)\widetilde{K}(D))[1] = \sign(S(D)) = w(D)\,.\qedhere
\]
\end{proof}
	
Now, everything is ready to conclude the proof of Theorem~\ref{thm}.
	
\begin{proof}[Proof of Theorem~\ref{thm}~(iii)]
The classical signature of an arbitrary knot~$K$ satisfies~$|\sigma(K)| \le 2g(K)$, with~$g(K)$ the genus of~$K$. If~$K$ is definite, then the opposite inequality holds by definition, yielding the equality~$|\sigma(K)| = 2g(K)$.
On the other hand, it is well-known that the span of the Alexander polynomial~$\Delta_K$ is bounded above by~$2g(K)$
(see e.g.~\cite[Proposition~6.13]{lickorish_introduction_1997}),
so~$\Delta_K$ has at most~$2g(K)$ roots counted with multiplicity. Moreover, the Levine-Tristram signature
of a knot vanishes close to~$\omega=1$ (see e.g.~\cite{gilmer_signature_2016}). Also, it can only have discontinuities at roots of~$\Delta_K$, and the jumps are bounded by the multiplicity of the corresponding roots (\cite[Theorem~2]{matumoto_signature_1977}; see also~\cite{gilmer_signature_2016}).
It follows that, since~$K$ is a definite knot, all the roots of the Alexander polynomial lie on the unit circle and all the jumps of the signature at discontinuities are maximal. The Levine-Tristram signature is therefore uniquely determined by the zeros of the Alexander polynomial (in~$S^1$).
		
Now, let~$D$ be a diagram for~$K$. By Theorem~\ref{thm}~(ii), we have
\[
\sign(\tau_D[0]) - w(D) =2 \sigma(K)
\]
and by Lemma~\ref{proposition_signature_writhe}, we have
\[
\sign(\tau_D[1]) - w(D) = 0\,.
\]
As above, by the bounds in Lemma~\ref{proposition_kashaev_jumps}, this forces all the jumps of~$\sign(\tau_D[x])$ to be maximal.
Since the jumps of~$\sign(\tau_D[x])$ are exactly the double of the jumps of~$\sigma_K(\omega)$, we therefore get the equality
\[
\sgn(\tau_D[x]) - w(D) = 2\sigma_K(\omega)
\]
for all~$\omega\in S^1\setminus\{1\}$, concluding the proof.
\end{proof}

\begin{rem}
\label{rem:disc}
Note that Theorem~\ref{thm}~(iii) holds for the classical Levine-Tristram signature, i.e. also when the signature function is evaluated at its points of discontinuity, but it does not hold for the two-sided average signature function.
\end{rem}

\begin{rem}
\label{rem:ext}
		In fact, our proof of Kashaev's signature conjecture applies to a wider class of knots,
		namely any knot~$K$ such that~$|\sigma(K)|$ is equal to the number of roots of~$\Delta_K$ on~$S^1$. This is for instance the case if~$K$ is the boundary of a Murasugi sum of two Seifert surfaces
		with symmetric, positive definite Seifert form, as proved by Liechti~\cite[Proposition~5.6]{liechti_spectra_2017}. This includes, in particular, all positive arborescent Hopf plumbings.
	\end{rem}

\bibliographystyle{plain}
\bibliography{Bibliography}
	
\end{document}